\documentclass{gOPT2e}

\usepackage{amsmath,amssymb,amsthm}
\usepackage[colorlinks,urlcolor=red]{hyperref}

\newcommand {\R} {\mathbb R}
\newcommand {\N} {\mathbb N}
\newcommand{\bx}{\bar x}
\newcommand{\by}{\bar y}
\newcommand{\bc}{\bar c}
\newcommand{\la}{\lambda}
\newcommand{\de}{\delta}
\newcommand {\B} {\mathbb B}
\newcommand {\cl} {{\rm cl}\,}
\newcommand{\set}[1]{\left\{#1\right\}}
\newcommand {\dom} {{\rm dom}\,}
\newcommand {\Int} {{\rm int}\,}
\newcommand{\ds}{\displaystyle}

\renewcommand\theenumi{\roman{enumi}}

\def\lsc{lower semicontinuous}

\newtheorem{thm}{Theorem}

\newtheorem{cor}[thm]{Corollary}

\newtheorem{prop}[thm]{Proposition}
\theoremstyle{definition}
\newtheorem{defn}[thm]{Definition}
\theoremstyle{remark}
\newtheorem{rem}[thm]{Remark}

\begin{document}
\title{{\itshape Borwein--Preiss Vector Variational Principle}}
\author{Alexander Y. Kruger$^*$
\thanks{$^*$Corresponding author. Email: a.kruger@federation.edu.au}
\quad
Somyot Plubtieng$^\dag$
\quad
Thidaporn Seangwattana$^\dag$
\\\vspace{6pt}
$^*${\em{Centre for Informatics and Applied Optimisation,\\
Faculty of Science and Technology,
Federation University Australia,\\
POB 663, Ballarat, Vic, 3350, Australia}}
\\\vspace{6pt}
$^\dag${\em{Department of Mathematics, Faculty of Science, Naresuan University,\\
Phitsanulok 65000, Thailand}}
\\\received{Dedicated to Professor Franco Giannessi in celebration of his 80th birthday\\
and Professor Diethard Pallaschke in celebration of his 75th birthday}}
\date{}
\maketitle

\begin{abstract}
This article extends to the vector setting the results of our previous work Kruger et al. 
(2015) which refined and slightly strengthened the metric space version of the
Borwein--Preiss variational principle due to Li and Shi, \emph{J. Math. Anal. Appl.} 246(1), 308--319 (2000).
We introduce and characterize two seemingly new natural concepts of $\epsilon$-minimality, one of them dependant on the chosen element in the ordering cone and the fixed ``gauge-type'' function.
\end{abstract}
\begin{classcode}
49J52; 49J53; 58C06
\end{classcode}
\section{Introduction}

Given an ``almost minimal'' point of a function, a variational principle guaranties the existence of another point and a suitably perturbed function for which this point is (strictly) minimal and provides estimates of the (generalized) distance between the points and also the size of the perturbation.
They are among the main tools in optimization theory and various branches of analysis.

The principles differ mainly in terms of the class of perturbations they allow.
The perturbation guaranteed by the conventional Ekeland variational principle \cite{Eke74} is nonsmooth even if the underlying space is a smooth Banach space and the function is everywhere Fr\'echet differentiable.
In contrast, the \emph{Borwein--Preiss variational principle} \cite{BorPre87} (originally formulated in the Banach space setting) works with a special class of perturbations determined by the norm; when the space is \emph{smooth} (i.e., the norm is Fr\'echet differentiable away from the origin), the perturbations are smooth too.
Because of that, the Borwein--Preiss variational principle is often referred to as the \emph{smooth variational principle}.
It has found numerous applications and paved the way for a number of smooth principles \cite{Phe89,DevGodZiz93.2,FabHajVan96,DevIva97, IofTik97,FabMor98, LoeWan01,Pen01,Geo05}.
Among the known extensions of the Borwein--Preiss variational principle, we mention the work by Li and Shi \cite[Theorem 1]{LiShi00}, where the principle was extended to metric spaces (of course at the expense of losing the smoothness), covering also the conventional Ekeland variational principle.

Since mid-1980s (cf. Loridan \cite{Lor84}, N\'emeth \cite{Nem86}, Khanh \cite{Kha89.2}), a considerable amount of research has been devoted to extending variational principles (mainly the conventional one due to Ekeland \cite{Eke74}) to vector-valued functions and, more recently, to set-valued mappings.
We refer the reader to books \cite{GopRiaTamZal03,CheHuaYan05,KhaTamZal15} and several recent articles \cite{Ha05, BaoMor07,
LiuNg11,TamZal11, KhaQuy13} for a good account of the results and various approaches.

Along with various scalarization techniques, generalized vector metrics have been used \cite{Kha89.2,CheHuaYan05,GutJimNov08}.
Some authors have considered
directional
and more general set perturbations \cite{Tam92,GopHenTam95,FinQuaTro03, GopRiaTamZal03,BedPrz07,BedZag09,KhaQuy10, BedZag10,LiuNg11,SitPlu12,KhaQuy13,KhaTamZal15}.
Using the latter approach,
Bednarczuk and Zagrodny \cite{BedZag10} obtained recently an extension of the Borwein--Preiss variational principle to vector-valued functions.

This article extends to the vector setting the results of our previous work \cite{KruPluSea1} which refined and slightly strengthened the metric space version of the Borwein--Preiss variational principle due to Li and Shi \cite{LiShi00}.

The structure of the article is as follows.
In the next, preliminary section, we discuss and compare various concepts of (approximate) minimality, boundedness, lower semicontinuity arising in the vector settings and relevant for the model studied in the current article.
In particular, we introduce in Definition~\ref{D11} two seemingly new natural concepts of $\epsilon$-minimality, one of them dependant on the chosen element in the ordering cone and the fixed ``gauge-type'' function (the term introduced by Li and Shi \cite{LiShi00}; cf. Definition~\ref{D2}) on the source space.
This seems to be the weakest $\epsilon$-minimality property ensuring the conclusions of the variational principle proved in Section~\ref{S3}.
A comparison of these concepts with other approximate minimality and lower boundedness properties is provided.

Section~\ref{S3} is dedicated to the Borwein--Preiss vector variational principle.
It is established in Theorem~\ref{main}.
In its statement and proof, we exploit and sharpen an idea of Li and Shi \cite{LiShi00} which allows elements of a sequence $\{\delta_{i}\}_{i=0}^\infty\subset\R_+$ involved in the statement of the theorem to be either all strictly positive or equal zero starting from some number.
This technique allowed Li and Shi to obtain an extension of both Borwein--Preiss and Ekeland variational principles.
In the final Section~\ref{S4}, we discuss the main result proved in Section~\ref{S3} and formulate a series of remarks and several corollaries.

Our basic notation is standard, cf. \cite{RocWet98,BorZhu05,DonRoc14}.
$X$ and $Y$ stand for either metric or normed spaces.
A metric or a norm in any space are denoted by $d(\cdot,\cdot)$ or $\|\cdot\|$, respectively.
$\N$ denotes the set of all positive integers.


\section{Level sets, minimality, boundedness, lower semicontinuity}

In this section, $f$ is a function from a metric space $X$ to a normed vector space $Y$ and $C$ is a pointed convex cone in $Y$, i.e., $C+C\subset C$, $\alpha C\subset C$ for $\alpha\in(0,\infty)$, and $C\cap(-C)=\{0\}$.
This cone is going to play the role of an ordering cone in $Y$.

Given a point $\by\in Y$, we can consider lower and upper \emph{$\by$-sublevel} sets of $f$ with respect to $C$:
\begin{gather*}
S_{\by}^\le(f):=\set{x\in X\mid \by-f(x)\in C},
\quad
S_{\by}^\ge(f):=\set{x\in X\mid f(x)-\by\in C}.
\end{gather*}
Obviously
\begin{gather*}
S_{\by}^\le(f)
\cap
S_{\by}^\ge(f)=\set{x\in X\mid f(x)=\by}.
\end{gather*}

If $\by=f(\bx)$ for some $\bx\in X$, we will write $S^\le(f,\bx)$ and $S^\ge(f,\bx)$ instead of $S_{f(\bx)}^\le(f)$ and $S_{f(\bx)}^\ge(f)$, respectively.
It is easy to check that a point $\bx$ is a (\emph{Pareto}) \emph{minimal} (\emph{efficient}) point of $f$ if and only if $f(x)=f(\bx)$ for all $x\in S^\le(f,\bx)$, i.e.,
$S^\le(f,\bx)\subset S^\ge(f,\bx)$.
Similarly, $\bx$ is a \emph{maximal point} of $f$ if and only if
$S^\ge(f,\bx)\subset S^\le(f,\bx)$.

We will also use the notation
\begin{gather*}
S_{\epsilon}^\ge(f,\bx):=\set{x\in X\mid
f(x)-f(\bx)\in C+\epsilon\B},
\end{gather*}
where $\epsilon\ge0$.
Obviously, $S_{0}^\ge(f,\bx)=S^\ge(f,\bx)$.

Following \cite[Theorem 1]{LiShi00} and \cite[Definition 2.5.1]{BorZhu05}, we are going to employ in the rest of the article the following concept of \emph{gauge-type} function.

\begin{defn}\label{D2}
Let $(X,d)$ be a metric space. We say that a continuous function $\rho : X\times X\rightarrow [0,\infty]$ is a gauge-type function if
\begin{enumerate}
\item
$\rho(x,x) = 0$ for all $x \in X,$
\item
for any $\epsilon > 0$ there exists $\delta > 0$ such that, for all $y,z \in X$, inequality $\rho(y,z) \leq \delta$ implies $d(y,z) < \epsilon.$
\end{enumerate}
\end{defn}

Given a gauge-type function $\rho$ on $X$ and a point $\bar c\in C$, we will consider the set (cf. the \emph{lower sector of $\bx$} \cite[p.~956]{KhaQuy13})
\begin{gather*}
S_{\rho,\bar c}^\le(f,\bx):=\set{x\in X\mid f(\bx)-f(x)-\rho(x,\bx)\bar c\in C}.
\end{gather*}
Obviously, $S_{\rho,\bar c}^\le(f,\bx)\subset S^\le(f,\bx)$.

\begin{defn}\label{D11}
Let $\epsilon\ge 0$.
We say that $\bx\in X$ is
\begin{itemize}
\item
an $\epsilon$-minimal point of $f$ if $S^\le(f,\bx)\subset S_{\epsilon}^\ge(f,\bx)$;
\item
an $\epsilon$-minimal point of $f$ with respect to $\rho$ and $\bar c$ if $S_{\rho,\bar c}^\le(f,\bx)\subset S_{\epsilon}^\ge(f,\bx)$.
\end{itemize}
\end{defn}
Any $\epsilon$-minimal point of $f$ is obviously an $\epsilon$-minimal point of $f$ with respect to any gauge-type function $\rho$ and any $\bar c\in C$.

\begin{prop}\label{P12}
Let $\epsilon\ge 0$.
$\bx\in X$ is an $\epsilon$-minimal point of $f$ if and only if
\begin{gather}\label{sil}
f(X)\cap(f(\bx)-C)\subset f(\bx)+C+\epsilon\B.
\end{gather}
\end{prop}

\begin{proof}
Let
$\bx\in X$ be an $\epsilon$-minimal point of $f$ and
$y\in f(X)\cap(f(\bx)-C)$, i.e., $y=f(x)$ for some $x\in X$ and $f(x)\in f(\bx)-C$ or, equivalently,
$f(\bx)-f(x)\in C$, i.e., $x\in S^\le(f,\bx)$.
By Definition~\ref{D11}, $x\in S_{\epsilon}^\ge(f,\bx)$, i.e., $f(x)-f(\bx)\in C+\epsilon\B$ and $y=f(x)\in f(\bx)+C+\epsilon\B$.

Conversely, let \eqref{sil} holds true and $x\in S^\le(f,\bx)$, i.e., $f(\bx)-f(x)\in C$ or, equivalently, $f(x)\in f(\bx)-C$.
By \eqref{sil}, $f(x)\in f(\bx)+C+\epsilon\B$.
Thus, $f(x)-f(\bx)\in C+\epsilon\B$, i.e., $x\in S_{\epsilon}^\ge(f,\bx)$.
\end{proof}

Recall that function $f$ is called \emph{level $C$-bounded} \cite{BedZag10} at $\bx\in X$ if
\begin{gather}\label{ag}
f(X) \cap (f(\bx)-C) \subset M + C
\end{gather}
for some bounded subset $M \subset Y$.
Obviously, $f$ is level $C$-bounded at $\bx$ if it is \emph{$C$-(lower) bounded} \cite{GopHenTam95}: $f(X)\subset M + C$ for some bounded subset $M \subset Y$; cf. \cite{BorZhu93,BedZag09,BedKru12.2}.

\begin{prop}\label{P12+}
$f$ is level $C$-bounded at $\bx$ if and only if $\bx\in X$ is an $\epsilon$-mi\-nimal point of $f$ for some $\epsilon\ge 0$.
\end{prop}

\begin{proof}
If $\bx$ is an $\epsilon$-minimal point of $f$ for some $\epsilon\ge 0$, then, by Proposition~\ref{P12}, condition \eqref{ag} is satisfied with the bounded set $M:=f(\bx)+\epsilon\B$.

Conversely, if condition \eqref{ag} is satisfied with some bounded set $M$, then there exists an $\epsilon\ge0$ such that $M\subset f(\bx)+\epsilon\B$, and \eqref{ag} implies \eqref{sil}, which, thanks to Proposition~\ref{P12}, means that $\bx$ is an $\epsilon$-minimal point of $f$.
\end{proof}

\begin{rem}
Thanks to Proposition~\ref{P12+}, the nonnegative number $\epsilon$ in Definition~\ref{D11} (and condition \eqref{sil}) provides a quantitative characterization of the level $C$-bo\-undedness of $f$ at $\bx$.
\end{rem}

Several other conditions can be used for defining/characterizing $\epsilon$-mini\-mality.
Here are some examples.
\begin{gather}\label{silw}
f(X)\subset f(\bx)+C+\epsilon\B,
\end{gather}
$\epsilon$-mini\-mality in the sense of Tanaka (cf. \cite{GutJimNov06,Bon09}):
\begin{gather}\label{silw2}
f(X)\cap(f(\bx)-C)\subset f(\bx)+\epsilon\B,
\end{gather}
$\epsilon$-mini\-mality in the direction $\bar c\in C\setminus\{0\}$  \cite{Lor84,Kut79} (cf. \cite{Tam92,BedPrz07,
Bon09,GutJimNov12,GopRiaTamZal03,CheHuaYan05}):
\begin{gather*}
f(X)\cap(f(\bx)-C\setminus\{0\}-\epsilon \bar c)=\emptyset.
\end{gather*}

\begin{prop}\label{P13}
Let $\epsilon\ge 0$ and $\bx\in X$.
\begin{enumerate}
\item
If either \eqref{silw} or \eqref{silw2} holds true, then
$\bx$ is an $\epsilon$-minimal point of $f$.
\item
If $\bx$ is an $\epsilon$-minimal point of $f$, then, for any $\bar c\in\Int C$,
\begin{gather*}
f(X)\cap(f(\bx)-C\setminus\{0\}-\xi \bar c)=\emptyset
\end{gather*}
where
$\xi=\epsilon/d(\bar c,X\setminus C)$.
\end{enumerate}
\end{prop}

\begin{proof}
(i) follows from comparing conditions \eqref{silw} and \eqref{silw2} with \eqref{sil}, thanks to Proposition~\ref{P12}.

(ii) Suppose that $f(X)\cap(f(\bx)-C\setminus\{0\}-\xi \bar c)\ne\emptyset$ for some $\bar c\in\Int C$ and $\xi=\epsilon/d(\bar c,X\setminus C)$, i.e., there is an $x\in X$ such that $f(x)\in f(\bx)-C\setminus\{0\}-\xi \bar c$ and consequently $f(x)\in f(\bx)-C$.
Then, by \eqref{sil}, $f(x)\in f(\bx)+C+\epsilon\B$.
Hence, $(f(\bx)-C\setminus\{0\}-\xi \bar c)\cap(f(\bx)+C+\epsilon\B)\ne\emptyset$.
It follows that
$(C+\bar c+(\epsilon/\xi)\B)\cap(-C\setminus\{0\})\ne\emptyset$.
By the assumption, $\bar c+(\epsilon/\xi)\B\subset C$, and consequently, $C+\bar c+(\epsilon/\xi)\B\subset C$.
Hence, $C\cap(-C\setminus\{0\})\ne\emptyset$, which is impossible
since $C$ is a pointed cone, a contradiction.
\end{proof}

\begin{defn}\label{D14}
Let $\bar c\in C$.
Function $f$ is called $C$-\lsc\ with respect to $\bar c$ at $x\in X$ if, for any $\{x_k\}\subset X$, $\{\epsilon_k\}\subset\R$ and $y\in Y$ such that $y-f(x_k)-\epsilon_k\bar c\in C$ ($k=1,2,\ldots$), $x_k\to x$ and $\epsilon_k\to0$ as $k\to\infty$, it holds $y-f(x)\in C$.

We say that $f$ is $C$-\lsc\ with respect to $\bar c$ on a subset $U\subset X$ if it is $C$-\lsc\ with respect to $\bar c$ at all $x\in U$.
In the case $U=X$, we say that $f$ is \lsc\ with respect to $\bar c$.
\end{defn}

Note that the defined above concept of $C$-lower semicontinuity with respect to $\bar c$ differs from that of $(\bar c,C)$-lower semicontinuity from \cite[p. 186]{CheHuaYan05}.

\begin{prop}\label{P15}
Let $\bar c\in C$.
$f$ is $C$-\lsc\ with respect to $\bar c$ if and only if, for any $y\in Y$ and any continuous function $g:X\to\R$, the set
\begin{align}\label{P15-1}
S:=\left\{x \in X\mid y- f(x)- g(x)\bar c\in C\right\}
\end{align}
is closed.
\end{prop}

\begin{proof}
Suppose $f$ is $C$-\lsc\ with respect to $\bar c$, $y\in Y$, $g:X\to\R$ is continuous, and let $x_k\in S$ and $x_k\to x$.
By the definition of the set $S$, $y'-f(x_k)-\epsilon_k\bar c\in C$, where $y':=y-g(x)\bar c$ and $\epsilon_k:=g(x_k)-g(x)\to0$ as $k\to\infty$.
By Definition~\ref{D14}, $y-f(x)-g(x)\bar c =y'-f(x)\in C$, i.e., $x\in S$.

Conversely, suppose the set $S$ is closed for any $y\in Y$ and any continuous function $g:X\to\R$, and consider arbitrary sequences $\{x_k\}\subset X$ and $\{\epsilon_k\}\subset\R$ and a point $y\in Y$ such that $y-f(x_k)-\epsilon_k\bar c\in C$ ($k=1,2,\ldots$), $x_k\to x$ and $\epsilon_k\to0$ as $k\to\infty$.
Passing to subsequences if necessary, suppose that all elements of $\{x_k\}$ are different.
By Tietze extension theorem, there exists a continuous function $g:X\to\R$ such that $g(x_k)=\epsilon_k$ ($k=1,2,\ldots$).
Then $x_k\in C$ ($k=1,2,\ldots$), $g(x)=0$ and, since set $S$ is closed, $y-f(x)\in C$.
\end{proof}

Similar to the corresponding property considered by Isac \cite{Isa97} (cf. \cite[p.~188]{CheHuaYan05}),
the $C$-lower semicontinuity with respect to $\bar c$ occupies an intermediate position between the $C$-lower semicontinuity and the closedness of lower sublevel sets.

Recall that $f$ is called \emph{$C$-\lsc} \cite{FinQuaTro03,GopRiaTamZal03} at $\bx\in X$ if for every neighbourhood $V$ of $f(\bx)$ there exists a neighbourhood $U$ of $\bx$ such that $f(U)\subset V+C$.
We say that $f$ is $C$-\lsc\ if it is $C$-\lsc\ at all $\bx\in X$.
Obviously, if $f$ is continuous at $\bx$, it is $C$-\lsc\ at this point.

\begin{prop}\label{P16}
Let $\bar c\in\Int C$.
\begin{enumerate}
\item
If $f$ is $C$-\lsc, then it is $C$-\lsc\ with respect to $\bar c$.
\item
If $f$ is $C$-\lsc\ with respect to $\bar c$, then, for any $y\in Y$, the lower sublevel set $S_y^\le$ is closed.
\end{enumerate}
\end{prop}

\begin{proof}
(i) Suppose that $f$ is $C$-\lsc, $x\in X$, $y\in Y$, $\epsilon>0$, and $U$ is a neighbourhood of $x$ such that $f(U)\subset f(x)+(\epsilon/2)\B+C$.
Let $\{x_k\}\subset X$, $\{\epsilon_k\}\subset\R$, $y-f(x_k)-\epsilon_k\bar c\in C$ ($k=1,2,\ldots$), $x_k\to x$, and $\epsilon_k\to0$ as $k\to\infty$.
Then, for all sufficiently large $k$, we have $\|\epsilon_k\bar c\|<\epsilon$, $x_k\in U$, and consequently, $f(x_k)-f(x)\in(\epsilon/2)\B+C$.
Hence, $y-f(x)\in\epsilon\B+C$.
Since $C$ is closed and $\epsilon$ is arbitrary, it follows that $y-f(x)\in C$, i.e., $f$ is $C$-\lsc\ with respect to $\bar c$.

(ii) Suppose that $f$ is $C$-\lsc\ with respect to $\bar c$, $y\in Y$, $\{x_k\}\subset S_y^\le$, i.e., $y-f(x_k)\in C$, and $x_k\to x\in X$ as $k\to\infty$.
It follows from Definition~\ref{D14} with $\epsilon_k=0$ for all $k$ that $y-f(x)\in C$, i.e., $x\in S_y^\le$.
Hence, $S_y^\le$ is closed.
\end{proof}

Replacing the assumption of $C$-lower semicontinuity of $f$ with respect to $\bar c$ in Proposition~\ref{P15} by the stronger property of $C$-lower semicontinuity allows one to partially strengthen its conclusion.

\begin{prop}\label{P11}
Let $\bar c\in C$.
If $f$ is $C$-\lsc, then the set \eqref{P15-1}
is closed for any $y\in Y$ and any \lsc\ function $g:X\to\R\cup\{+\infty\}$.
\end{prop}

\begin{proof}
Suppose $f$ is $C$-\lsc, $y\in Y$, $g:X\to\R$ is \lsc, and let $x_k\in S$ and $x_k\to x$.
Without loss of generality, $g(x_k)\to\alpha\ge g(x)$.
By the definition of the set $S$, $y'-f(x_k)-\epsilon_k\bar c\in C$, where $y':=y-\alpha\bar c$ and $\epsilon_k:=g(x_k)-\alpha\to0$ as $k\to\infty$.
By the definition of $C$-lower semicontinuity, $y'-f(x)\in C$, and consequently,
$$y-f(x)-g(x)\bar c =y'-f(x)+(\alpha-g(x))\bar c\in C+(\alpha-g(x))\bar c\subset C,$$
i.e., $x\in S$.
\end{proof}

\section{Borwein--Preiss vector variational principle}\label{S3}

In this section, we extend to vector-valued functions the metric space version of the Borwein--Preiss variational principle due to Li and Shi \cite{LiShi00} (cf. \cite{BorZhu05}).

The theorem below involves sequences indexed by $i\in\N$.
The set of all indices is subdivided into two groups: with $i<N$ and $i\ge N$ where $N$ is an `integer' which is allowed to be infinite: $N\in\N\cup\{+\infty\}$.
If $N=+\infty$, then the first subset of indices is infinite, while the second one is empty.
This trick allows us to treat the cases of a finite and infinite sets of indices within the same framework.
Another convention in this section concerns summation over an empty set of indices: $\sum_{k=0}^{-1}a_k=0$.

The next theorem presents a vector version of the Borwein--Preiss variational principle.

\begin{thm}\label{main}
Let $X$ be a complete metric space, Y a normed vector space, $C$ a pointed closed convex cone in $Y$, $\bar c\in\Int C$ and let a function $f : X \rightarrow Y$ be $C$-\lsc\ with respect to $\bar c$.
Suppose that $\rho$ is a gauge-type function, $\epsilon>0$, $\{\epsilon_{i}\}_{i=1}^\infty$ and $\{\delta_{i}\}_{i=0}^\infty$ are sequences such that
\begin{itemize}
\item
$\epsilon_{i}>0$ for all $i\in\N$ and $\epsilon_{i}\downarrow0$ as $i\to\infty$;
\item
$\de_{i}>0$ for all $i<N$ and $\de_{i}=0$ for all $i\ge N$, where $N\in\N\cup\{+\infty\}$.
\end{itemize}
If $x_0\in X$ is an $\epsilon$-minimal point of $f$ with respect to $\de_0\rho$ and $\bar c$,
then there exist a point $\bar x\in X$ and a sequence $\{x_{i}\}_{i=1}^\infty\subset X$ such that $x_i\to\bx$ as $i\to\infty$ and
\begin{enumerate}
\item
$\ds\rho(\bx,x_0) \le \frac{\epsilon}{\de_0 d(\bar c,X\setminus C)}$;
\item
$\ds\rho(\bx,x_{i}) <\epsilon_i$ $(i=1,2,\ldots)$;
\item
if $N=+\infty$, then the
series $\sum\limits_{i=0}^{\infty} \delta_{i}\rho(\bar x,x_{i})$ is convergent and
\begin{equation}\label{su11}
f(x_0)- f(\bar x)-\left(\sum\limits_{i=0}^{\infty} \delta_{i}\rho(\bar x,x_{i})\right)\bar c\in C;
\end{equation}
otherwise the
series $\sum_{i=N-1}^{\infty}\rho(x_{i+1},x_{i})$ is convergent and
\begin{multline}\label{su11-2}
f(x_{0})-f(\bx)
- \Biggl(\sum_{i=0}^{N-2}\delta_{i} \rho(\bx,x_{i})
\\
+ \delta_{N-1}\sup_{n\ge N-1}\left(\sum_{i=N-1}^{n-1} \rho(x_{i+1},x_{i})+ \rho(\bx,x_{n})\right)\Biggr)\bar c\in C;
\end{multline}
\item
for any $x \in X\setminus\{\bar x\}$, there exists an $m_0\ge N$ such that, for all $m\ge m_0$,\\
if $N=+\infty$, then,
\begin{equation}\label{su10}
f(\bar x)+ \left(\sum\limits_{i=0}^{\infty}\delta_{i}\rho(\bar x,x_{i})\right)\bar c - f(x)- \left(\sum\limits_{i=0}^{m} \delta_{i}\rho(x,x_{i})\right)\bar c \not\in C;
\end{equation}
otherwise,
\begin{multline}\label{us10}
f(\bx)+\left(\sum_{i=0}^{N-2}\delta_{i} \rho(\bx,x_{i})+ \delta_{N-1}\sup_{n\ge m} \left(\sum_{i=m}^{n-1}\rho(x_{i+1},x_{i})+ \rho(\bx,x_{n})\right)\right)\bar c
\\
- f(x)-\left(\sum_{i=0}^{N-2}\delta_{i}\rho(x,x_{i})
+ \delta_{N-1}\rho(x,x_{m})\right)\bar c\notin C.
\end{multline}
\end{enumerate}
\end{thm}

\begin{proof}
(i) and (ii).
Since $C$ is a pointed convex cone with $\Int C\ne\emptyset$,
there exists an element $y^{*} \in Y^{*}$ such that $\|y^*\|=1$ and $\langle y^{*},c\rangle \ge 0$ for all $c \in C$.
Set $\lambda:=\langle y^{*},\bar c\rangle$.
Since $\bar c\in\Int C$, we have $\lambda\ge d(\bar c,X\setminus C)$.

We define sequences $\{x_{i}\}$ and $\{S_{i}\}$ inductively.
Set
\begin{equation}\label{su8}
S_{0} := \{x \in X\mid f(x_{0})-f(x)-\de_0\rho(x,x_{0})\bar c \in C\}.
\end{equation}
Obviously, $x_{0}\in S_{0}$.
Since $f$ is $C$-\lsc\ with respect to $\bar c$, by Proposition~\ref{P15}, subset $S_{0}$ is closed: it is sufficient to take $y:=f(x_{0})$ and $g(x):=\de_0\rho(x,x_{0})$.
For any $x \in S_{0}$, we have
$$\langle y^{*},f(x_{0}) - f(x)\rangle\ge \lambda\de_0\rho(x,x_{0}).$$
At the same time, by Definition~\ref{D11}, $f(x)-f(x_{0})\in C+\epsilon\B$.
Hence,
$$\langle y^{*},f(x) - f(x_{0})\rangle\ge-\epsilon.$$
It follows from the last two inequalities that
\begin{equation}\label{su6}
\rho(x,x_{0}) \le\frac{\epsilon}{\lambda} \le \frac{\epsilon}{\de_0d(\bar c,X\setminus C)}.
\end{equation}

For $i=0,1,\ldots$, denote $j_i:=\min\{i,N-1\}$, i.e., $j_i$ is the largest integer $j\le i$ such that $\de_j>0$.
Let $i\in\N$ and
suppose $x_{0}, \ldots,x_{i-1}$ and $S_{0},\ldots,S_{i-1}$ have been defined.
We choose $x_{i} \in S_{i-1}$ such that
\begin{multline}\label{su2}
\langle y^{*},f(x_{i})\rangle +\lambda\sum_{k=0}^{j_i-1}\delta_{k}\rho(x_{i},x_{k})\\
< \inf_{x\in S_{i-1}}\left(\langle y^{*},f(x)\rangle + \lambda\sum_{k=0}^{j_i-1}\delta_{k}\rho(x,x_{k})\right) + \lambda\delta_{j_i}\epsilon_i
\end{multline}
and define
\begin{multline}\label{su4}
S_{i} := \Biggl\{x \in S_{i-1}\mid f(x_{i})- f(x)
\\
-\left(\sum_{k=0}^{j_i-1}\delta_{k} (\rho(x,x_{k})-\rho(x_{i},x_{k}))+ \delta_{j_i}\rho(x,x_{i})\right)\bar c\in C\Biggr\}.
\end{multline}
Obviously, $x_{i}\in S_{i}$.
Since $f$ is $C$-\lsc\ with respect to $\bar c$, by Proposition~\ref{P15}, subset $S_{i}$ is closed: it is sufficient to take $y:=f(x_{i})$ and $g(x):=\sum_{k=0}^{j_i-1}\delta_{k} (\rho(x,x_{k})-\rho(x_i,x_{k})) + \delta_{j_i}\rho(x,x_{i})$.
For any $x \in S_{i}$, we have
$$\langle y^{*},f(x_{i}) - f(x)\rangle+ \lambda\sum_{k=0}^{j_i-1}\delta_{k} (\rho(x_{i},x_{k})-\rho(x,x_{k})) -\lambda\delta_{j_i}\rho(x,x_{i})\ge0,$$
and consequently, making use of \eqref{su2},
\begin{multline}\label{su7}
\rho(x,x_{i})\le\frac{1}{\lambda\delta_{j_i}} \Biggl(\langle y^{*},f(x_{i})\rangle+ \lambda\sum_{k=0}^{j_i-1}\delta_{k}\rho(x_{i},x_{k})
\\
-\Bigl(\langle y^{*},f(x)\rangle +\lambda\sum_{k=0}^{j_i-1} \delta_{k}\rho(x,x_{k})\Bigr)\Biggr)<\epsilon_i.
\end{multline}
We can see that, for all $i\in\N$, subsets $S_{i}$ are nonempty and closed, $S_{i}\subset S_{i-1}$, and $\sup_{x\in S_i}\rho(x,x_{i}) \rightarrow 0$ as $i\to\infty$.
Since $\rho$ is a gauge-type function, we also have $\sup_{x\in S_i}d(x,x_{i}) \rightarrow 0$ and consequently,
${\rm diam} (S_{i}) \rightarrow 0.$
Since $X$ is complete, $\cap_{i=0}^{\infty}S_{i}$ contains exactly one point; let it be $\bx$.
Hence, $\rho(\bx,x_{i}) \rightarrow 0$ and $x_{i} \rightarrow \bar x$ as $i \rightarrow \infty$.
Thanks to \eqref{su6} and \eqref{su7}, $\bx$ satisfies (i) and (ii).
\medskip

Before proceeding to the proof of claim (iii), we prepare several building blocks which are going to be used when proving claims (iii) and (iv).

Let integers $m$, $n$ and $i$ satisfy $0\le m\le i<n$.
Since $x_{i+1}\in S_i$ and $\bx\in S_n$, it follows from \eqref{su8} (when $i=0$) and \eqref{su4} that
\begin{gather}\label{su3-}
f(x_{i})- f(x_{i+1})
-\left(\sum_{k=0}^{j_i-1}\delta_{k} (\rho(x_{i+1},x_{k})-\rho(x_{i},x_{k}))+ \delta_{j_i}\rho(x_{i+1},x_{i})\right)\bar c\in C,
\\\label{su3-+}
f(x_{n})- f(\bx)
- \left(\sum_{k=0}^{j_n-1}\delta_{k} (\rho(\bx,x_{k})-\rho(x_{n},x_{k}))+ \delta_{j_n}\rho(\bx,x_{n})\right)\bar c\in C.
\end{gather}
We are going to add together inclusions \eqref{su3-} from $i=m$ to $i=n-1$ and inclusion \eqref{su3-+}.
Depending on the value of $N$, three cases are possible.

\underline{If $N>n$}, then $j_i=i$ and $j_n=n$.
Adding inclusions \eqref{su3-} from $i=m$ to $i=n-1$, we obtain
\begin{equation*}
f(x_{m})- f(x_{n})
-\left(\sum_{k=0}^{n-1}\delta_{k} \rho(x_{n},x_{k})-\sum_{k=0}^{m-1}\delta_{k} \rho(x_{m},x_{k})\right)\bar c\in C.
\end{equation*}
Adding the last inclusion and inclusion \eqref{su3-+}, we arrive at
\begin{equation}\label{su3.3-}
f(x_{m})- f(\bx)
-\left(\sum_{k=0}^{n}\delta_{k} \rho(\bx,x_{k})-\sum_{k=0}^{m-1}\delta_{k} \rho(x_{m},x_{k})\right)\bar c\in C.
\end{equation}

\underline{If $N\le m$}, then $j_i=N-1$ and $j_n=N-1$.
Adding inclusions \eqref{su3-} from $i=m$ to $i=n-1$, we obtain
\begin{gather*}
f(x_{m})- f(x_{n})
-\left(\sum_{k=0}^{N-2}\delta_{k} (\rho(x_{n},x_{k})-\rho(x_{m},x_{k}))+ \delta_{N-1}\sum_{k=m}^{n-1}\rho(x_{k+1},x_{k})\right)\bar c\in C.
\end{gather*}
Adding the last inclusion and inclusion \eqref{su3-+}, we arrive at
\begin{multline}\label{su3.4-}
f(x_{m})- f(\bx)
- \Biggl(\sum_{k=0}^{N-2}\delta_{k} (\rho(\bx,x_{k})-\rho(x_{m},x_{k}))
\\
+ \delta_{N-1}\left(\sum_{k=m}^{n-1}\rho(x_{k+1},x_{k})+ \rho(\bx,x_{n})\right)\Biggr)\bar c\in C.
\end{multline}

\underline{If $m<N\le n$}, we add inclusions \eqref{su3-} separately from $i=m$ to $i=N-1$ and from $i=N$ to $i=n-1$ and obtain, respectively,
\begin{gather*}
f(x_{m})- f(x_{N})
-\left(\sum_{k=0}^{N-1}\delta_{k} \rho(x_{N},x_{k})-\sum_{k=0}^{m-1}\delta_{k} \rho(x_{m},x_{k})\right)\bar c\in C,
\\
f(x_{N})- f(x_{n})
-\left(\sum_{k=0}^{N-2}\delta_{k} (\rho(x_{n},x_{k})-\rho(x_{N},x_{k}))+ \delta_{N-1}\sum_{k=N}^{n-1}\rho(x_{k+1},x_{k})\right)\bar c\in C.
\end{gather*}
Adding the last two inclusions and inclusion \eqref{su3-+} together, we obtain
\begin{multline}\label{su3.3-+}
f(x_{m})- f(\bx)
- \Biggl(\sum_{k=0}^{N-2}\delta_{k} \rho(\bx,x_{k})-\sum_{k=0}^{m-1}\delta_{k} \rho(x_{m},x_{k})
\\
+ \delta_{N-1}\left(\sum_{k=N-1}^{n-1}\rho(x_{k+1},x_{k})+ \rho(\bx,x_{n})\right)\Biggr)\bar c\in C.
\end{multline}

(iii) When $N=+\infty$, we set $m=0$ in the inclusion \eqref{su3.3-}:
\begin{equation}\label{su9}
f(x_{0})- f(\bx)- \left(\sum_{k=0}^{n}\delta_{k}\rho(\bx,x_{k})\right)\bar c\in C.
\end{equation}
Since $C$ is a pointed cone and $\bar c\ne0$, it holds $(-\bar c+r\B)\cap C=\emptyset$ for some $r>0$, and consequently, $(-s_n\bar c+s_nr\B)\cap C=\emptyset$, where $s_n:=\sum_{k=0}^{n}\delta_{k}\rho(\bx,x_{k})$.
It follows from \eqref{su9} that $s_n r\le\|f(x_{0})- f(\bx)\|$ for all $n\in\N$.
This implies that the series $\sum_{k=0}^{\infty} \delta_k\rho(\bar x,x_k)$ is convergent and, thanks to \eqref{su9}, condition \eqref{su11} holds true.

When $N<+\infty$, we set $m=0$ and take $n=N-1$ in the inclusion \eqref{su3.3-} and any $n\ge N$ in the inclusion \eqref{su3.3-+}:
\begin{align}\label{su9-50}
f(x_{0})- f(\bx)
&-\left(\sum_{k=0}^{N-1}\delta_{k} \rho(\bx,x_{k})\right)\bar c\in C,
\\\notag
f(x_{0})-f(\bx)
&- \Biggl(\sum_{k=0}^{N-2}\delta_{k} \rho(\bx,x_{k})
\\\label{su9-5}
&+ \delta_{N-1} \left(\sum_{i=N-1}^{n-1}\rho(x_{i+1},x_{i})+ \rho(\bx,x_{n})\right)\Biggr)\bar c\in C.
\end{align}
As above, for some $r>0$ and any $n>N$, it holds $(-\delta_{N-1}s_n\bar c+\delta_{N-1}s_nr\B)\cap C=\emptyset$, where $s_n:=\sum_{i=N-1}^{n-1}\rho(x_{i+1},x_{i})$.
It follows from \eqref{su9-5} that
$$\delta_{N-1}s_nr\le\|f(x_{0})-f(\bx)\|+ \sum_{k=0}^{N-2}\delta_{k} \rho(\bx,x_{k})+\delta_{N-1}\rho(\bx,x_{n})\|\bc\|.$$
Since $\rho(\bx,x_{n})\to0$ as $n\to\infty$, this implies that the series $\sum_{i=N-1}^{\infty} \rho(x_{i+1},x_{i})$ is convergent.
Combining the two inclusions \eqref{su9-50} and \eqref{su9-5} produces estimate \eqref{su11-2}.

(iv) For any $x \neq \bar x,$ there exists an $m_0\in\N$ such that $x\notin S_{m}$ for all $m\ge m_0$.
By \eqref{su4}, this means that
\begin{equation}\label{su5}
f(x_{m})- f(x)
- \left(\sum_{k=0}^{j_m-1}\delta_{k} (\rho(x,x_{k})-\rho(x_{m},x_{k}))+ \delta_{j_m}\rho(x,x_{m})\right)\bar c\notin C.
\end{equation}
Depending on the value of $N$, we consider two cases.

\underline{If $N=+\infty$}, then $j_m=m$.
Since the series $\sum_{k=0}^{\infty} \delta_{k}\rho(\bar x,x_{k})$ is convergent and $C$ is closed, we can pass in \eqref{su3.3-} to the limit as $n\to\infty$ to obtain
\begin{equation*}
f(x_{m})+ \left(\sum_{k=0}^{m-1}\delta_{k}\rho(x_{m},x_{k})\right)\bar c - f(\bar x)-\left(\sum_{k=0}^{\infty} \delta_{k}\rho(\bar x,x_{k})\right)\bar c \in C.
\end{equation*}
Comparing the last inclusion with \eqref{su5}, we arrive at condition (\ref{su10}).

\underline{If $N<\infty$}, we can take $m_0\ge N$.
Then $j_m=N-1$ and it follows from \eqref{su3.4-} that
\begin{multline*}
f(x_{m})- f(\bx)
- \Biggl(\sum_{k=0}^{N-2}\delta_{k} (\rho(\bx,x_{k})-\rho(x_{m},x_{k}))
\\
+ \delta_{N-1}\sup_{n\ge m} \left(\sum_{k=m}^{n-1}\rho(x_{k+1},x_{k})+ \rho(\bx,x_{n})\right)\Biggr)\bar c\in C.
\end{multline*}
Comparing the last inclusion with \eqref{su5}, we arrive at \eqref{us10}.
\end{proof}

\section{Comments and Corollaries}\label{S4}
In this section, we discuss the main result proved in Section~\ref{S3} and formulate a series of remarks and several corollaries.
A number of remarks related to the scalar version of Theorem~\ref{main} in \cite{KruPluSea1} are also applicable to the more general setting considered here.

\begin{rem}\label{R18}
1. If $N<\infty$, in the proof of part (iv) of Theorem~\ref{main} one can also consider the case $m_0<N$.
Then, for $m_0\le m<N$, one has $j_m=m$ and it follows from \eqref{su3.3-+} that
\begin{multline*}
f(x_{m})- f(\bx)
- \Biggl(\sum_{k=0}^{N-2}\delta_{k} \rho(\bx,x_{k})-\sum_{k=0}^{m-1}\delta_{k} \rho(x_{m},x_{k})
\\
+ \delta_{N-1}\sup_{n\ge N} \left(\sum_{k=N-1}^{n-1}\rho(x_{k+1},x_{k})+ \rho(\bx,x_{n})\right)\Biggr)\bar c\in C.
\end{multline*}
Comparing the last inclusion with \eqref{su5}, one arrives at
\begin{multline}\label{su10+}
f(\bx)
+\left(\sum_{k=0}^{N-2}\delta_{k} \rho(\bx,x_{k})+ \delta_{N-1}\sup_{n\ge N} \left(\sum_{k=N-1}^{n-1}\rho(x_{k+1},x_{k})+ \rho(\bx,x_{n})\right)\right)\bar c
\\
-f(x)-\left(\sum_{k=0}^{m}\delta_{k} \rho(x,x_{k})\right)\bar c\notin C.
\end{multline}
This estimate compliments \eqref{us10}.

2. The role of the assumption of $C$-lower semicontinuity of function $f$ with respect to $\bar c$ in Theorem~\ref{main} is to ensure the closedness of the sets  \eqref{su8} and \eqref{su4}.
For that purpose, one can use the following weaker (but in general more difficult to verify) condition: for any finite collection $\{x_0,\ldots,x_n\}\in X\; (0\le n<N)$, the set
\begin{align*}
\left\{x \in X\mid f(x_{n})+ \sum_{k=0}^{n-1}\delta_{k}\rho(x_{n},x_{k})\bar c - f(x)- \sum_{k=0}^{n}\delta_{k}\rho(x,x_{k})\bar c\in C\right\}
\end{align*}
is closed.

Thanks to Proposition~\ref{P16}(i), it is sufficient to assume that $f$ is $C$-\lsc.
In the latter case, thanks to Proposition~\ref{P11}, one can weaken the assumption of continuity of $\rho$ in Definition~\ref{D2} of a gauge-type function.
As in \cite{LiShi00}, it is sufficient to assume that $\rho$ is \lsc\ in its first argument.

3. Instead of $\epsilon$-minimality of $x_0$ with respect to $\de_0\rho$ and $\bar c$, it is sufficient to assume in Theorem~\ref{main} that $x_0\in X$ is simply an $\epsilon$-minimal point of $f$.
In this case, sequence $\{\delta_{i}\}_{i=0}^\infty$ can be scaled; in particular, one can assume that $\sum_{i=0}^\infty\delta_{i}=1$.
Thanks to Proposition~\ref{P12+}, the assumption of $\epsilon$-mi\-ni\-mality of $x_0$ can be replaced by that of level $C$-bo\-un\-dedness of $f$ at $x_0$ (as in \cite[Theorem~3.1]{BedZag10}).
In this case, one would have to drop estimate (i).
One can also use stronger (thanks to Proposition~\ref{P13}) concepts of $\epsilon$-minimality given by conditions \eqref{silw} or \eqref{silw2}.

4. Assumption $\bar c\in\Int C$ can be replaced by a weaker condition $\bar c\in C\setminus\{0\}$.
In this case, condition (i) becomes meaningless and should be dropped.

5. Given a number $\la>0$, one can talk in Theorem~\ref{main} about $\epsilon$-minimality with respect to $\epsilon/\la$ and $\bar c$ (or just $\epsilon$-minimality) and formulate a more conventional form of the variational principle with $\de_0=1$ and conditions (i), \eqref{su11} and \eqref{su10} replaced, respectively, with the following ones:
\begin{enumerate}
\item [(i$^\prime$)]
$\ds\rho(\bx,x_0) \le \frac{\la}{d(\bar c,X\setminus C)}$,
\end{enumerate}
\begin{gather}\tag{\ref{su11}$^\prime$}
f(x_0)- f(\bar x)-\frac{\epsilon}{\la}\left(\sum\limits_{i=0}^{\infty} \delta_{i}\rho(\bar x,x_{i})\right)\bar c\in C,
\\\tag{\ref{su10}$^\prime$}
f(\bar x)+ \frac{\epsilon}{\la}\left(\sum\limits_{i=0}^{\infty}\delta_{i}\rho(\bar x,x_{i})\right)\bar c - f(x)- \frac{\epsilon}{\la} \left(\sum\limits_{i=0}^{\infty}\delta_{i}\rho(x,x_{i})\right)\bar c \not\in C
\end{gather}
and similar amendments in conditions \eqref{su11-2}, \eqref{us10} and \eqref{su10+}.

6. A similar result (though only for the case $N=+\infty$ and without estimate (i)) was established in \cite[Theorem~3.1]{BedZag10} in a more general setting where, instead of a single element $\bar c\in\Int C$, a closed convex subset $D\subset C$ with the property $0\notin\cl(D+C)$ is used.
When $D=\{\bar c\}$ and $\bar c\in\Int C$ (or just $\bar c\in C\setminus\{0\}$; cf. item~4 above), the assumptions of Theorem~\ref{main} are weaker.
In particular, $Y$ is not assumed reflexive Banach space, $C$ is not assumed \emph{normal} in the sense of Krasnoselski et al. \cite{KraLifSob85}, the sequences $\{\epsilon_{i}\}$ and $\{\delta_{i}\}$ are not assumed to belong to $(0,1)$, and the series $\sum_{i=1}^{\infty} \epsilon_{i}$ is not assumed convergent.
\end{rem}

\begin{cor}\label{C3}
Suppose all the assumptions of Theorem~\ref{main} are satisfied, and $N=\infty$. Then
\begin{equation}\label{ss10}
f(\bar x)+ \left(\sum\limits_{i=0}^{\infty}\delta_{i}\rho(\bar x,x_{i})\right)\bar c - f(x)- \left(\sum\limits_{i=0}^{\infty}\delta_{i}\rho(x,x_{i})\right)\bar c \notin C
\end{equation}
for all $x\in X\setminus\{\bar x\}$ such that the series $\sum_{i=0}^{\infty}\delta_{i}\rho(x,x_{i})$ is convergent.
\end{cor}

\begin{proof}
Condition \eqref{ss10} is a direct consequence of \eqref{su10} since $\sum_{i=0}^{m}\delta_{i}\rho(x,x_{i})\le \sum_{i=0}^{\infty}\delta_{i}\rho(x,x_{i})$.
\end{proof}

\begin{cor}\label{C4}
Suppose all the assumptions of Theorem~\ref{main} are satisfied, and $N<\infty$. Then
\begin{gather}\label{C4-11}
f(x_{0})-f(\bx)
- \left(\sum_{i=0}^{N-1}\delta_{i} \rho(\bx,x_{i})
\right)\bar c\in C,
\\\label{C4-12}
f(x_{0})-f(\bx)
- \left(\sum_{i=0}^{N-2}\delta_{i} \rho(\bx,x_{i})
+ \delta_{N-1}\sum_{i=N-1}^{\infty} \rho(x_{i+1},x_{i})\right)\bar c\in C,
\end{gather}
and, for any $x \in X\setminus\{\bar x\}$, there exists an $m_0\ge N$ such that, for all $m\ge m_0$,
\begin{align}\notag
f(\bx)+\Biggl(\sum_{i=0}^{N-2}&\delta_{i} \rho(\bx,x_{i})+ \delta_{N-1}\rho(\bx,x_{m})\Biggr)\bar c
\\\label{C4-21}
&- f(x)-\left(\sum_{i=0}^{N-2}\delta_{i}\rho(x,x_{i})
+ \delta_{N-1}\rho(x,x_{m})\right)\bar c\notin C,
\\\notag
f(\bx)+\Biggl(\sum_{i=0}^{N-2}&\delta_{i} \rho(\bx,x_{i})+ \delta_{N-1}\sum_{i=m}^{\infty} \rho(x_{i+1},x_{i})\Biggr)\bar c
\\\label{C4-22}
&- f(x)-\left(\sum_{i=0}^{N-2}\delta_{i}\rho(x,x_{i})
+ \delta_{N-1}\rho(x,x_{m})\right)\bar c\notin C,
\end{align}
and consequently,
\begin{multline}\label{ss10-2}
f(\bx)+\left(\sum_{i=0}^{N-2}\delta_{i} \rho(\bx,x_{i})\right)\bar c
- f(x)
\\
-\Biggl(\sum_{i=0}^{N-2}\delta_{i}\rho(x,x_{i})
+ \delta_{N-1}\rho(x,\bx)\Biggr)\bar c\notin\Int C
\quad\mbox{for all}\quad
x \in X,
\end{multline}
where $\bx$ and $\{x_{i}\}_{i=1}^\infty$ are a point and a sequence guaranteed by Theorem~\ref{main}.
\end{cor}

\begin{proof}
Conditions \eqref{C4-11} and \eqref{C4-12} correspond, respectively, to setting $n=N-1$ and letting $n\to\infty$ under the $\sup$ in condition \eqref{su11-2}.
Similarly, conditions \eqref{C4-21} and \eqref{C4-22} correspond, respectively, to setting $n=m$ and letting $n\to\infty$ under the $\sup$ in condition \eqref{us10}.
Condition \eqref{ss10-2} is obviously true when $x=\bx$.
When $x\ne\bx$, it results from passing to the limit as $m\to\infty$ in any of the conditions \eqref{C4-21} and \eqref{C4-22} thanks to the continuity of $\rho$.
\end{proof}

\begin{rem}
1. Unlike the limiting condition \eqref{ss10} in Corollary~\ref{C3}, the original condition \eqref{su10} in Theorem~\ref{main} is applicable also when the series $\sum_{i=0}^{\infty}\delta_{i}\rho(x,x_{i})$ is divergent.

2. In general, condition \eqref{su11-2} is stronger than each of the conditions \eqref{C4-11} and \eqref{C4-12} which are independent.
A similar relationship is true between conditions \eqref{us10}, \eqref{C4-21} and \eqref{C4-22}.
As observed in \cite{KruPluSea1}, thanks to Corollary~\ref{C4}, in the scalar case Theorem~\ref{main} strengthens \cite[Theorem~1]{LiShi00}.

3. Conditions \eqref{ss10}, \eqref{C4-21} and \eqref{ss10-2} can be interpreted as a kind of minimality at $\bx$ of a perturbed function.
When $N=+\infty$, the conclusion of Corollary~\ref{C3} says that
\begin{gather}\label{con}
(f+g)(\bx)-(f+g)(x)\notin C
\quad\mbox{for all}\quad
x\in\dom g\setminus\{\bar x\},
\end{gather}
where the perturbation function $g$ is defined by $$g(x):=\left(\sum_{i=0}^{\infty} \delta_{i}\rho(x,x_{i})\right)\bc$$
for all $x\in X$
such that the series $\sum_{i=0}^{\infty}\delta_{i}\rho(x,x_{i})$ is convergent.

When $N<+\infty$, condition \eqref{C4-21} in Corollary~\ref{C4} is equivalent to \eqref{con} with
$$g(x):=\left(\sum_{i=0}^{N-2}\delta_{i}\rho(x,x_{i})+ \delta_{N-1}\rho(x,x_m)\right)\bc.$$
Note that in this case $\dom g=X$.

In contrast, condition \eqref{ss10-2} represents a weaker form of minimality:
$$(f+g)(\bx)-(f+g)(x)\notin\Int C
\quad\mbox{for all}\quad
x\in X,$$
where the perturbation function $g$ is defined by
$$g(x):=\left(\sum_{i=0}^{N-2}\delta_{i}\rho(x,x_{i})+ \delta_{N-1}\rho(x,\bx)\right)\bc.$$
Thanks to the next proposition, if function $\rho$ possesses the triangle inequality, the latter condition can be strengthened.
\end{rem}

Recall that a function $\rho:X\times X\to\R$ possesses the \emph{triangle inequality} if $\rho(x_1,x_3)\le \rho(x_1,x_2)+\rho(x_2,x_3)$ for all $x_1,x_2,x_3\in X$.

\begin{prop}\label{C6}
Along with conditions \eqref{C4-11}--\eqref{C4-21}, consider the following one:
\begin{gather}\label{C6-4}
f(\bx)+\Biggl(\sum_{i=0}^{N-2}\delta_{i} \rho(\bx,x_{i})\Biggr)\bar c
- f(x)-\left(\sum_{i=0}^{N-2}\delta_{i}\rho(x,x_{i})
+ \delta_{N-1}\rho(x,\bx)\right)\bar c\notin C.
\end{gather}
If function $\rho$ possesses the triangle inequality,
then \eqref{C4-12} $\Rightarrow$ \eqref{C4-11} and \eqref{C4-22} $\Rightarrow$ \eqref{C4-21} $\Rightarrow$ \eqref{C6-4}.
\end{prop}

\begin{proof}
For any $m,n\in\N$ with $m<n$, we have
$$
\rho(\bx,x_{m})\le\rho(\bx,x_{n})+ \sum_{i=m}^{n-1} \rho(x_{i+1},x_{i}),
$$
and consequently, passing to the limit as $n\to\infty$,
\begin{align*}
\rho(\bx,x_{m})\le\sum_{i=m}^{\infty} \rho(x_{i+1},x_{i}).
\end{align*}
Hence,  \eqref{C4-12} $\Rightarrow$ \eqref{C4-11} and \eqref{C4-22} $\Rightarrow$ \eqref{C4-21}.
Condition \eqref{C6-4} follows from \eqref{C4-21} thanks to the inequality $\rho(x,x_{m})\le\rho(x,\bx)+\rho(\bx,x_{m})$.
\end{proof}

\begin{cor}\label{C6-0}
Suppose all the assumptions of Theorem~\ref{main} are satisfied, $N<+\infty$, and function $\rho$ possesses the triangle inequality.
Then condition \eqref{C6-4} holds true for all $x\in X\setminus\{\bx\}$, where $\bx$ and $\{x_{i}\}_{i=1}^\infty$ are a point and a sequence guaranteed by Theorem~\ref{main}.
\end{cor}

\begin{proof}
The statement is a consequence of Corollary~\ref{C4} thanks to Proposition~\ref{C6}.
\end{proof}

The next two statements are consequences of Theorem~\ref{main} when $N=+\infty$ and $N=1$, respectively, and $\rho$ is of a special form.
The first one corresponds to the case $N=+\infty$, $X$ a Banach space and $\rho(x_1,x_2):= \|x_1-x_2\|^p$ where $p>0$.

\begin{cor}\label{T4}
Let $(X,\|\cdot\|)$ be a Banach space, Y a normed vector space, $C$ a pointed closed convex cone in $Y$, $\bar c\in\Int C$ and let function $f : X \rightarrow Y$ be $C$-\lsc\ with respect to $\bar c$.
Suppose that $\la$, $p$, $\epsilon$, $\epsilon_{i}$ $(i=1,2,\ldots)$, $\delta_{i}$ $(i=0,1,\ldots)$ are positive numbers and $\epsilon_{i}\downarrow0$ as $i\to\infty$.

If $x_0\in X$ is an $\epsilon$-minimal point of $f$ and $\de_0$ is such that $x_0$ is also a $(\de_0\epsilon)$-minimal point of $f$,
then there exist a point $\bar x\in X$ and a sequence $\{x_{i}\}_{i=1}^\infty\subset X$ such that $x_i\to\bx$ as $i\to\infty$ and
\begin{enumerate}
\item
$\ds\|\bar x-x_{0}\|\le \frac{\la}{d(\bar c,X\setminus C)^{1/p}}$;
\item
$\ds\|\bar x-x_{i}\|<\epsilon_i$ $(i=1,2,\ldots)$;
\item
$\ds f(x_0)- f(\bar x)- \frac{\epsilon}{\la^p}\left(\sum\limits_{i=0}^{\infty} \delta_{i}\|\bar x-x_{i}\|^p\right)\bar c\in C$;
\item
for any $x \in X\setminus\{\bar x\}$, there exists an $m\in\N$ such that
$$\ds f(\bar x)+ \frac{\epsilon}{\la^p} \left(\sum\limits_{i=0}^{\infty}\delta_{i}\|\bar x-x_{i}\|^p\right)\bar c - f(x)- \frac{\epsilon}{\la^p} \left(\sum\limits_{i=0}^{m}\delta_{i} \|x-x_{i}\|^p\right)\bar c \notin C,$$
and consequently,
\begin{gather}\label{ty}
\ds f(\bar x)+ \frac{\epsilon}{\la^p} g(\bx)\bar c
- f(x)- \frac{\epsilon}{\la^p} g(x)\bar c \notin C
\;\;\mbox{for all}\;\;
x \in\dom g\setminus\{\bar x\},
\end{gather}
where $g(x):=\sum_{i=0}^{\infty}\delta_{i} \|x-x_{i}\|^p$.
\end{enumerate}
\end{cor}
\begin{proof}
Set $\rho(x_1,x_2):= \|x_1-x_2\|^p$, $x_1,x_2\in X$.
It is easy to check that $\rho$ is a gauge-type function.
Set $\epsilon':=\epsilon\de_0$, $\epsilon_i':=\epsilon_i^{p}$ $(i=1,2,\ldots)$, $\de_i':=(\epsilon/\la^p)\de_i$ $(i=0,1,\ldots)$.
Then $\epsilon_i'\downarrow0$ as $i\to\infty$ and $\epsilon'/\de_0'=\la^p$.
The conclusion follows from Theorem~\ref{main} with $\epsilon'$, $\epsilon_i'$ and $\de_i'$ in place of $\epsilon$, $\epsilon_i$ and $\de_i$, respectively.
\end{proof}

If $X$ is Fr\'echet smooth, $p>1$, and $\sum_{i=0}^{\infty}\delta_{i}<\infty$, then $\dom g=X$ in \eqref{ty} and $g$ is everywhere Fr\'echet differentiable, i.e., we have an example of a smooth variational principle of Borwein--Preiss type.
In the scalar case,
Corollary~\ref{T4} extends and strengthens the conventional theorem of Borwein and Preiss \cite[Theorem~2.6]{BorPre87}, \cite[Theorem~2.5.3]{BorZhu05} along several directions.

1) Condition (iv) is in general stronger than merely condition \eqref{ty}.
It can still be meaningful even at those points $x$ where the series $\sum_{i=0}^{\infty}\delta_{i}\|x-x_{i}\|^p$ is divergent.
If this series is convergent for all $x\in X$, then the conditions are equivalent.

2) Apart from the requirement of the $(\de_0\epsilon)$-minimality of $\bx$, no other restrictions are imposed on the positive numbers $\delta_{i}$, $i=0,1,\ldots$

3) Corollary~\ref{T4} does not exclude the ``tight'' case when $\epsilon$ is the minimal number such that $x_0\in X$ is an $\epsilon$-minimal point of $f$.
In the latter case, the requirement of the $(\de_0\epsilon)$-minimality of $\bx$ is equivalent to $\de_0\ge1$.
This still allows to chose positive numbers $\delta_{i}$, $i=1,2,\ldots$, such that $\sum_{i=0}^{\infty}\delta_{i}<\infty$ if necessary.

When the $\epsilon$-minimality of $\bx$ is not tight (in the scalar case, this means that $f(x_{0})-\inf_{X}f<\epsilon$), then one can choose $\de_0<1$ such that $x_0$ is a $(\de_0\epsilon)$-minimal point of $f$ and positive numbers $\delta_{i}$, $i=1,2,\ldots$, such that $\sum_{i=0}^{\infty}\delta_{i}=1$.

4) In the scalar case, the conventional theorem of Borwein and Preiss assumes the strict inequality $f(x_{0})-\inf_{X}f<\epsilon$ and claims the existence of positive numbers $\delta_{i}$, $i=0,1,\ldots$, satisfying $\sum_{i=0}^{\infty}\delta_{i}=1$ such that the scalar versions of conditions (iii) and \eqref{ty} hold true (under the appropriate choice of a point $\bar x$ and a sequence $\{x_{i}\}_{i=1}^\infty$).
Under the same assumption, Corollary~\ref{T4} guarantees the same (or stronger) conclusions for all positive numbers $\delta_{i}$, $i=0,1,\ldots$, with $\delta_0<1$ satisfying the requirement of the $(\de_0\epsilon)$-minimality of $\bx$ and such that $\sum_{i=0}^{\infty}\delta_{i}=1$.

5) The power index $p$ in (iii) and (iv) is an arbitrary positive number and can be less than 1.
\medskip

The next statement corresponds to $N=1$ and $\rho$ being a distance function.
It generalizes the Ekeland variational principle.

\begin{cor}\label{T3}
Let $(X,d)$ be a complete metric space, Y a normed vector space, $C$ a pointed closed convex cone in $Y$, $\bar c\in\Int C$ and let $f : X \rightarrow Y$ be $C$-\lsc\ with respect to $\bar c$.
Suppose $\la>0$ and $\epsilon>0$.
If $x_0\in X$ is an $\epsilon$-mi\-nimal point of $f$,
then there exists a point $\bar x\in X$ such that
\begin{enumerate}
\item
$\ds d(\bx,x_{0})\le \frac{\la}{d(\bar c,X\setminus C)}$;
\item
$\ds f(x_{0})-f(\bx)
-\frac{\epsilon}{\la}d(\bx,x_{0})\bar c\in C;$
\item
$\ds f(\bx)- f(x)-\frac{\epsilon}{\la}d(x,\bx)\bar c\notin C$
for all $x \in X\setminus\{\bar x\}$.
\end{enumerate}
\end{cor}
\begin{proof}
Set $\rho:=d$, $N=1$,
$\de_0:=\epsilon/\la$,
$\epsilon_i:=\epsilon/2^{i}$ and $\de_i:=0$ $(i=1,2,\ldots)$.
Then $\epsilon_i\downarrow0$ as $i\to\infty$ and $\epsilon/\de_0=\la$.
The conclusion follows from Theorem~\ref{main} and Corollary~\ref{C6}.
\end{proof}

\begin{rem}
1. Instead of $C$-lower semicontinuity of function $f$, it is sufficient to assume in Corollary~\ref{T3} that, for any $x\in X$, the set
\begin{align*}
\left\{u\in X\mid f(x)- f(u)-d(u,x)\bar c\in C\right\}
\end{align*}
is closed (cf. Remark~\ref{R18}.2).

2. One can try to use the estimates in Theorem~\ref{main} and its corollaries for developing a ``smooth'' theory of vector error bounds similar to the conventional theory based on the application of vector versions of the Ekeland variational principle (cf. \cite{BedKru12,BedKru12.2}).
\end{rem}
\section*{Acknowledgments}
The research was supported by the Australian Research Council, project DP110102011; Naresuan University, and Thailand Research Fund, the Royal Golden Jubilee Ph.D. Program.
\bibliographystyle{gOPT}
\bibliography{buch-kr,kruger,kr-tmp}
\end{document}